\documentclass[12pt,reqno]{amsart}

\usepackage{fullpage}

\usepackage{times}
\usepackage[T1]{fontenc}
\usepackage{mathrsfs}
\usepackage{latexsym}
\usepackage[dvips]{graphics}
\usepackage[dvips]{graphicx}
\usepackage{epsfig}
\usepackage{amsmath,amsfonts,amsthm,amssymb,amscd}
\input amssym.def
\input amssym.tex
\usepackage{color}
\usepackage{hyperref}
\usepackage{url}

\usepackage{verbatim}

\newcommand{\bburl}[1]{\textcolor{blue}{\url{#1}}}

\newcommand{\twocase}[5]{#1 \begin{cases} #2 & \text{{\rm #3}}\\ #4
&\text{{\rm #5}} \end{cases}  }

\DeclareMathOperator{\Tr}{Tr}

\DeclareMathOperator{\End}{End}

\newcommand{\Ha}{\mathbb{H}}

\newcommand{\fF}{\mathcal{F}}

\newcommand{\EE}{\mathcal{E}}
\newcommand{\nn}{\nonumber}

\DeclareMathOperator{\GOE}{GOE}
\DeclareMathOperator{\GUE}{GUE}
\DeclareMathOperator{\GSE}{GSE}

\newcommand{\hf}{{}_0F_1}
\newcommand{\foh}{\frac{1}{2}}

\newcommand{\fof}{\frac{1}{4}}

\newtheorem{thm}{Theorem}[section]

\newtheorem{lem}[thm]{Lemma}

\numberwithin{equation}{section}

\newtheorem{defn}[thm]{Definition}

\newcommand\be{\begin{equation}}
\newcommand\ee{\end{equation}}
\newcommand\bea{\begin{eqnarray}}
\newcommand\eea{\end{eqnarray}}
\newcommand\bi{\begin{itemize}}
\newcommand\ei{\end{itemize}}
\newcommand\ben{\begin{enumerate}}
\newcommand\een{\end{enumerate}}

\renewcommand{\leq}{\leqslant}
\renewcommand{\mod}[1]{
	{\ifmmode\text{\rm\ (mod~$#1$)}\else
		\discretionary{}{}{\hbox{ }}\rm(mod~$#1$)\fi}
}

\newcommand{\R}{\ensuremath{\mathbb{R}}}
\newcommand{\C}{\ensuremath{\mathbb{C}}}

\newcommand{\E}{\ensuremath{\mathbb{E}}}

\newsymbol\dnd 232D

\begin{document}

\title{Spherical Matrix Ensembles}

\author{Gene S. Kopp}
\address{
Department of Mathematics,
University of Michigan,
Ann Arbor, MI,
USA}
\email{\textcolor{blue}{\href{mailto:gkopp@umich.edu}{gkopp@umich.edu}}}

\author{Steven J. Miller}
\address{
Department of Mathematics and Statistics,
Williams College,
Williamstown, MA,
USA
}
\email{\textcolor{blue}{\href{mailto:sjm1@williams.edu, Steven.Miller.MC.96@aya.yale.edu}{sjm1@williams.edu, Steven.Miller.MC.96@aya.yale.edu}}}

\begin{abstract}
The spherical orthogonal, unitary, and symplectic ensembles (SOE/SUE/SSE) $S_\beta(N,r)$ consist of $N \times N$ real symmetric, complex hermitian, and quaternionic self-adjoint matrices of Frobenius norm $r$, made into a probability space with the uniform measure on the sphere. For each of these ensembles, we determine the joint eigenvalue distribution for each $N$, and we prove the empirical spectral measures rapidly converge to the semicircular distribution as $N \to \infty$.  In the unitary case ($\beta=2$), we also find an explicit formula for the empirical spectral density for each $N$.
\end{abstract}



\date{\today}

\thanks{Portions of this work were completed at the 2010 SMALL REU at Williams College; we thank our colleagues there, especially Murat Kolo$\breve{{\rm g}}$lu, and participants at the ICM Satellite Meeting in Probability \& Stochastic Processes in Bangalore, India, especially Arup Bose, for helpful conversations. The first named author was partially supported by NSF Grant DMS0850577, and the second named author was partially supported by NSF grants DMS0970067 and DMS1265673.}

\maketitle

\thispagestyle{empty}

\tableofcontents

\section{Introduction}

\noindent \texttt{After writing this paper, it was brought to our attention that the our spherical ensembles have been studied previously by physicists Ca\"{e}r and Delannay, under the name of \textit{fixed trace ensembles} \cite{CD, DC}.   Delannay and Ca\"{e}r prove results equivalent to our Theorem 1.4 (on\\ moments of spherical ensembles) for $\beta = 1, 2$ in Appendix B of \cite{DC}, and our Theorem 1.5 (on empirical spectral density) in Section 3.1 of \cite{DC}.  Our methods differ from those of Delannay and Ca\"{e}r: We use the method of moments, and they use the Laplace transform method.}

\ \\ \ \\

Random matrix theory has grown enormously in results and scope from Wishart's \cite{Wis} 1928 paper. Through the work of Wigner \cite{Wig1, Wig2, Wig3, Wig4, Wig5}, Dyson \cite{Dy1, Dy2}, and others, it now successfully models a variety of systems from the energy levels of heavy nuclei and zeros of $L$-functions \cite{Con,Por} to bus routes \cite{BBDS,KrSe}; see for example \cite{AGZ,For,Meh,MT-B} for detailed expositions of the general theory, and \cite{FM,Hay} for a review of some of the historical development and applications.

Much of the subject revolves around determining properties of the spectra of self-adjoint matrices. In this paper we concentrate on the density of states, though there has been significant progress in recent years on the spacings between normalized eigenvalues (see for example \cite{ERSY,ESY,TV1,TV2}). The limiting spectral measure has been computed for many ensembles of patterned or structured matrices, including band, circulant, Hankel and Toeplitz matrices, as well as adjacency matrices of $d$-regular graphs \cite{BBDS, BasBo1, BasBo2, BanBo, BCG, BHS1, BHS2, BM, BDJ, GKMN, HM, JMP, Kar, KKMSX, LW, McK, Me, Sch}. Many of these papers show that when a classical ensemble is modified by enforcing particular, usually linear, relations among the entries, new limiting behavior emerges.  We study \textit{spherical ensembles} introduced in \cite{MT-B} (see Research Project 16.4.7). While in general it is difficult to obtain closed-form expressions for spectral statistics of structured ensembles, we are able to exploit the symmetry of the spherical ensembles to give explicit formulas for the joint spectral distribution and empirical spectral density.  We prove a semicircular law and give numerical evidence of GUE spacing statistics.

As a set, the $N \times N$ \textbf{spherical unitary ensemble of radius $r$} consists of $N \times N$ Hermitian matrices
\be A\ = \ \left(
\begin{array}{cccc}
a_{11} & a_{12}+b_{12}\sqrt{-1} & \cdots & a_{1N}+b_{1N}\sqrt{-1} \\
a_{12}-b_{12}\sqrt{-1} & a_{22} & \cdots & a_{2N}+b_{2N}\sqrt{-1} \\
\vdots & \vdots & \ddots & \vdots \\
a_{1N}-b_{1N}\sqrt{-1} & a_{2N}-b_{2N}\sqrt{-1} & \cdots & a_{NN}
\end{array}
\right)\ee
satisfying the relation
\be
\sum_{i} a_{ii}^2 + 2\sum_{i<j} \left(a_{ij}^2+b_{ij}^2\right) \ = \  r^2,
\ee
or more succinctly, $\Tr(A^2)=r^2$.  This set of matrices becomes a probability space with the uniform measure on the (compact) ellipsoid defined by $\Tr(A^2)=r^2$. One interpretation of the name is that ${\rm Tr}(A^2) = r^2$ is a sphere in the Frobenius norm (see below).

The spherical ensemble may be understood to arise naturally from the problem of normalization.  Consider a Hermitian matrix $A$ chosen from a Gaussian unitary ensemble, with volume form $e^{-\Tr(A^2)/2}dA$ (up to a constant).  To study the limiting behavior as the dimension $N \to \infty$, it's standard to normalize the eigenvalues by a factor of $1/\sqrt{N}$.  With this normalization, the eigenvalue density converges to Wigner's semicircle,
\be
\twocase{f_{\rm \text{{\rm Wig}}}(x) \ := \ }{\frac{1}{2\pi} \sqrt{4 - x^2},}{if $|x| \le 1$,}{0,}{otherwise.}
\ee
One could instead normalize the matrices themselves before looking at the eigenvalues.  We can normalize a Hermitian matrix $A$ by sending it to $\frac{\sqrt{N}}{|A|}A$, where $|A|$ is the Frobenius norm (see \eqref{eq:frobeniusnorm}). This ``pre-normalized GUE'' is just the spherical unitary ensemble of radius $\sqrt{N}$.


We also consider spherical orthogonal and symplectic ensembles.  In all cases, we prove that the eigenvalue density converges to Wigner's semicircle as the dimension $N \to \infty$.  \emph{In the fixed-dimensional case, however, the empirical and joint spectral distributions exhibit new behavior.}  We provide explicit formulas and investigate their analytic properties.


\subsection{Definitions}

Let $V$ be a finite-dimensional real, complex, or quaternionic\footnote{A \textit{quaternionic vector space} is a free $(\Ha,\Ha)$-bimodule, where $\Ha$ denotes the Hamilton quaternions.  A finite-dimensional quaternionic vector space is necessarily isomorphic to some $\Ha^n$.} vector space with a Hermitian inner product $\langle -, - \rangle$.
For any $A \in \End(V)$, let $A^\ast \in \End(V)$ be the adjoint, so $\langle Av, w \rangle = \langle v, A^\ast w \rangle$.  The $\textbf{Frobenius inner product}$ on $\End(V)$ is defined by
\be
\langle A, B \rangle \ := \  \Tr(A B^\ast),
\ee
and induces the \textbf{Frobenius norm}
\be\label{eq:frobeniusnorm}
|A| \ := \  \sqrt{\langle A, A \rangle} \ = \  \sqrt{\Tr(A A^\ast)}.
\ee
When $A$ is written as a matrix (with respect to an orthonormal basis on $V$), the Frobenius norm of $A$ equals the square root of the sum of the norm squares of the entries of $A$.  If $A$ is self-adjoint, then $|A| = \sqrt{\Tr(A^2)}$ also equals the square root of the sum of the squares of the eigenvalues of $A$.

For our purposes, an $N \times N$ matrix ensemble is simply a subset of $M_N(K)$, where $K = \R, \C$ or $\Ha$, along with probability measure. For a matrix ensemble $\mathcal{E}$, the underlying set is also denoted by $\mathcal{E}$, and the associated volume form is denoted by $d_\mathcal{E}A$.  The ensembles we examine are all ensembles of self-adjoint matrices and thus have real eigenvalues, so we assume that in the discussion that follows.

If it exists, the \textbf{empirical spectral density} of the ensemble $\mathcal{E}$ is the unique (up to a measure zero set of values) real-valued function $f(x,\EE)$ such that
\be
\E_{A \in \EE}\left(\frac{1}{N}\sum_i g(\lambda_i(A))\right) \ = \  \int_{-\infty}^{\infty} g(x)f(x,\EE)\,dx
\ee
for all $g \in L^1(\R)$, where $\lambda_1(A) \leq \lambda_2(A) \leq \cdots \leq \lambda_N(A)$ are the eigenvalues of $A$.  The $k^{\rm th}$ moment of $f(x,\EE)$ is
\begin{align}
m_k(\mathcal{E}) \ :=   \E_{A \in \EE}\left(\frac{1}{N} \lambda_i(A)^k\right) \ = \  \frac{1}{N}\int_\mathcal{E} \Tr(A^k)d_\mathcal{E}A,
\end{align}
and its corresponding \textbf{characteristic function} is
\be
\phi(t,\EE) \ = \  \sum_{k=0}^\infty \frac{m_k(\mathcal{E})}{k!}(it)^k
\ee
and is related to the density by
\be
f(x,\EE) \ = \  \frac{1}{2\pi} \int_{-\infty}^{\infty} e^{-itx}\phi(t,\EE)\,dt.
\ee
Let $S_N$ be the symmetric group on $\{1,2,\ldots,N\}$.  The \textbf{joint spectral distribution} is the measure on the set $\R^N/S_N$ of $N$ unordered points induced by pushing forward the measure on $\EE$ along the map $A \mapsto \left(\lambda_1(A), \lambda_2(A), \ldots, \lambda_N(A)\right)$.   We consider it as a measure on $\R^N$ by  symmetrizing it and normalizing it to be a probability measure.

We can now carefully state the ensembles of interest in this paper.

\begin{defn}[Gaussian ensembles]
The \textbf{Gaussian orthogonal / unitary / symplectic ensembles (GOE/GUE/GSE)} are denoted $G_\beta(N,q)$ and consist of $N \times N$ real symmetric, complex Hermitian, and quaternionic self-adjoint matrices, made into a probability space with the density
\be
d_{G_\beta(N,q)}A \ := \  C_{\beta}(N,q)e^{-\beta q |A|^2/4}dA.
\ee
Here $\beta = 1, 2, 4$ for the $\GOE$, $\GUE$, and $\GSE$, respectively, $q$ is a positive real number, and $C_{\beta}(N,q)$ is chosen to give a probability density.
\end{defn}

\begin{defn}[Spherical ensembles]
The \textbf{spherical orthogonal / unitary / symplectic ensembles (SOE/SUE/SSE)} $S_\beta(N,r)$ consist of $N \times N$ real symmetric, complex hermitian, and quaternionic self-adjoint matrices of Frobenius norm $r$, made into a probability space with the uniform measure on the sphere.
\end{defn}

\subsection{Results}

For all spherical ensembles, we describe the joint eigenvalue distribution.
\begin{thm}\label{thm:joint} Let $\beta\in\{1,2,4\}$. The joint eigenvalue distribution for $S_\beta(N,r)$ (on vectors of eigenvalues $\lambda \in \R^N$)
is given by
\be\label{eq:formulajoint}
\rho(\lambda, S_\beta(N,r))d\lambda
\ = \  C''_\beta(N,r) |\Delta(\lambda)|^\beta \delta(r^2-|\lambda|^2)d\lambda.
\ee
Here, $C''_\beta(N,r)$ is a computable constant, $\delta$ is the Dirac delta, and
\be
\Delta(\lambda) \ = \  \prod_{1 \leq i < j \leq N} \left(\lambda_j - \lambda_i\right).
\ee
\end{thm}


The formula in \eqref{eq:formulajoint} can be used---at least theoretically---to describe any statistic for these ensembles, such as spacing density, largest spacing, $n$-level correlation, $n$-level density, et. cetera.  However, this is easier said than done.  The analysis even to compute the marginals
\begin{equation}
\rho^{(m)}((\lambda_1,\ldots,\lambda_m),S_\beta(N,r))\ :=\ \int_{-\infty}^\infty \cdots \int_{-\infty}^\infty \rho\left(\lambda,S_\beta(N,r)\right)\, d\lambda_{m+1} \cdots d\lambda_N
\end{equation}
was too cumbersome.  The remainder of our results concern the empirical spectral density
\begin{equation}
f\left(x,S_\beta(N,r)\right)\ =\ \rho^{(1)}((x),S_\beta(N,r))
\end{equation}
and its moments.

For all spherical ensembles, we are able to express the moments in terms of the moments of the corresponding Gaussian ensembles, which allows us to prove a semicircular law in the limit. Explicitly,

\begin{thm}\label{thm:momentsandsemicircle} Let $\beta \in \{1,2,4\}$, and let $q, r > 0$. The $k^{\rm th}$ moment of a spherical ensemble is related to the $k^{\rm th}$ moment of a Gaussian ensemble by
\be\label{eq:momentsandsemicircletheoremmomentformula}
m_k(S_\beta(N,r))
\ = \
\left(\frac{\beta}{4} q r^2\right)^{\frac{k}2} \frac{\Gamma\left(\frac{n}{2}\right)}{\Gamma\left(\frac{k+n}{2}\right)}
m_k(G_\beta(N,q)),
\ee
where $n = \dim_\R G_\beta(N,q) = \beta {N \choose 2} + N$.  Specializing to $r = \sqrt{N}$, the eigenvalue density of $S_\beta(N,\sqrt{N})$ has mean zero and variance one.  The density functions $f(x, S_\beta(N,\sqrt{N}))$ converge in measure to $f_{\rm \text{{\rm Wig}}}(x)$.
\end{thm}

In the Hermitian case, we are able to give explicit formulas for the characteristic function and density function.

\begin{thm}\label{thm:characteristicdensity}
Let $J_\alpha$ be a Bessel function of the first kind (see chapter 9 of \cite{AS}).  The characteristic function equals
\be
\phi(t,S_2(N,r)) \ = \  \frac{\Gamma(\frac{N^2}{2})}{N} \left(\frac{2}{rt}\right)^{\frac{N^2}{2}-1} \sum_{j=0}^{N-1} \frac{1}{j!} {N \choose j+1} J_{\frac{N^2}{2}+j-1}(rt) (-rt)^j,
\ee
and the spectral density is \begin{align}
f\left(x,S_2(N,r)\right)
&\ = \  \frac{1}{r^{N^2-2}N}\sum_{j=0}^{N-1} {N \choose j+1} \frac{\Gamma\left(\frac{N^2}{2}\right)}{2^{j}\sqrt{\pi}\Gamma\left(\frac{N^2-1}{2}+j\right) j!}\left(\frac{d}{dx}\right)^{2j}\left(r^2-x^2\right)^{\frac{N^2-3}{2}+j}\nn\\
&\ = \ \frac{1}{r}p_N\left(\frac{x}{r}\right)\left(1-\left(\frac{x}{r}\right)^2\right)^{\frac{1}{2}(N^2-2N-1)}
\end{align}
for a particular polynomial $p_N$.
\end{thm}

The paper is organized as follows. We derive some necessary lemmas in \S\ref{sec:preliminary}.  We prove Theorem \ref{thm:joint} in \S\ref{sec:proofoftheoremjoint}, Theorem \ref{thm:momentsandsemicircle} in \S\ref{sec:proofofomentsandsemicircle}, and Theorem \ref{thm:characteristicdensity} in \S\ref{sec:proofoftheoremcharacteristicdensity}.  We conclude by discussing spacing statistics in \S\ref{sec:spacings}.

\section{Preliminary Lemmas}\label{sec:preliminary}
\noindent
We calculate the joint and empirical spectral distributions of the spherical ensembles by relating them to the Gaussian ensembles, whose spectral distributions are known explicitly.  

\begin{lem}\label{lem:map}
The map\footnote{While this map is not defined at the zero matrix, this is immaterial for our investigations below since the zero matrix is a single point of mass $0$.} of measure spaces $A \mapsto \frac{r}{|A|} A$ defines a measure-preserving map from $G_\beta(N,q)$ to $S_\beta(N,r)$. In other words, if $f \in L^1(S_\beta(N,r))$, then
\be\label{eq:SbetaNrtoGbetaNq}
\int_{S_\beta(N,r)} f(A)d_{S_\beta(N,r)}A \ = \  \int_{G_\beta(N,q)} f\left(\frac{r}{|A|} A\right)d_{G_\beta(N,q)}A.
\ee
\end{lem}

\begin{proof}
Because the density function of a Gaussian ensemble depends only on the Frobenius norm, the mass of each ray from the origin is the same.  Collapsing each such ray to the point on the Frobenius sphere it intersects gives the uniform distribution on the sphere.
\end{proof}

The next lemma will be needed
in the proof of Theorem \ref{thm:momentsandsemicircle}.

\begin{lem}\label{lem:integration} Let $G$ be a degree $k$ homogeneous polynomial in $n$ variables with complex coefficients, $B$ a positive definite symmetric matrix, $|\cdot|$ the standard $2$-norm on $\R^n$, $\sqrt{\vec{x}^{\top} B \vec{x}}$ the norm induced by $B$, and $\alpha > 0$. Then
\be
\int_{\R^n} \frac{G(\vec{x})}{(\vec{x}^{\top} B \vec{x})^{k/2}} e^{-\alpha(\vec{x}^{\top} B \vec{x})/4}d\vec{x} \ = \  \left(\frac{\alpha}{4}\right)^{\frac{k}{2}}\frac{\Gamma\left(\frac{n}{2}\right)}{\Gamma\left(\frac{k+n}{2}\right)} \int_{\R^n} G(\vec{x}) e^{-\alpha \vec{x}^{\top} B \vec{x}/4}d\vec{x}.
\ee
\end{lem}

\begin{proof} We first do the special case when $B=I$. Denote by $d\vec{\theta}$ the volume element on $S^{n-1}$.
For any function $f : \R \to \R$ so that $G(\vec{x})f(|\vec{x}|)$ is integrable,
\begin{align}
\int_{\R^n} G(\vec{x})f(|\vec{x}|)d\vec{x}
&\ = \  \int_{S^{n-1}}\int_0^\infty G(r\vec{\theta})f(r)r^{n-1}drd\vec{\theta}\nn\\
&\ = \  \int_{S^{n-1}}\int_0^\infty G(\vec{\theta})f(r)r^{k+n-1}drd\vec{\theta}\nn\\
&\ = \  \left(\int_0^\infty f(r)r^{k+n-1}dr\right)\int_{S^{n-1}} G(\vec{\theta})d\vec{\theta}.
\end{align}
If $f(r) \ = \  e^{-\alpha r^2/4}r^\gamma$ for $\alpha > 0$, we have the $r$-integral is
\begin{align}
\int_0^\infty f(r)r^{k+n-1}dr
&\ = \  \int_0^\infty e^{-\alpha r^2/4} r^{\gamma +k+n-1}dr\nn\\
&\ = \  \int_0^\infty e^{-u}\sqrt{\frac{4u}{\alpha}}^{\gamma +k+n-1}\sqrt{\frac{1}{\alpha u}}du\nn\\
&\ = \  \frac{1}{2}\left(\frac{4}{\alpha}\right)^{\frac{\gamma +k+n}{2}}\Gamma\left(\frac{\gamma + k+n}{2}\right).
\end{align}
The cases $\gamma = -k$ and $\gamma = 0$ give
\begin{align}\label{eq:specialcaseBequalsI}
\int_{\R^n} \frac{G(\vec{x})}{|\vec{x}|^k} e^{-\alpha |\vec{x}|^2/4}d\vec{x} &\ = \  \frac{1}{2}\left(\frac{4}{\alpha}\right)^{\frac{n}{2}}\Gamma\left(\frac{n}{2}\right) \int_{S^{n-1}} G(\vec{\theta})d\vec{\theta},\nn\\
\int_{\R^n} G(\vec{x}) e^{-\alpha |\vec{x}|^2/4}d\vec{x}  &\ = \  \frac{1}{2}\left(\frac{4}{\alpha}\right)^{\frac{k+n}{2}}\Gamma\left(\frac{k+n}{2}\right)\int_{S^{n-1}} G(\vec{\theta})d\vec{\theta}.
\end{align}
Dividing these two equations gives the desired formula for the $B=I$ case.

The general case follows immediately from \eqref{eq:specialcaseBequalsI}. As $B$ is positive definite, we have $B = A^{\top}A$ for some full rank matrix $A$.  The linear substitution $\vec{y} = A\vec{x}$ yields the more general formula, as
\begin{align}
\int_{\R^n} \frac{G(\vec{x})}{\left(\vec{x}^{\top}B\vec{x}\right)^{k/2}} e^{-\alpha \left(\vec{x}^{\top}B\vec{x}\right)/4}d\vec{x}
&\ = \  \int_{\R^n} \frac{G(\vec{x})}{|A\vec{x}|^{k}} e^{-\alpha |A\vec{x}|^2/4}d\vec{x}\nn\\
&\ = \  \int_{\R^n} \frac{G\left(A^{-1}\vec{y}\right)}{|\vec{y}|^{k}} e^{-\alpha |\vec{y}|^2/4}|\det{A}|^{-1}d\vec{y}\nn\\
&\ = \  \left(\frac{\alpha}{4}\right)^{\frac{k}{2}}\frac{\Gamma\left(\frac{n}{2}\right)}{\Gamma\left(\frac{k+n}{2}\right)} \int_{\R^n} G\left(A^{-1}\vec{y}\right)e^{-\alpha |\vec{y}|^2/4}|\det{A}|^{-1}d\vec{y}\nn\\
&\ = \  \left(\frac{\alpha}{4}\right)^{\frac{k}{2}}\frac{\Gamma\left(\frac{n}{2}\right)}{\Gamma\left(\frac{k+n}{2}\right)} \int_{\R^n} G(\vec{x})e^{-\alpha \left(\vec{x}^{\top}B\vec{x}\right)/4}d\vec{x}.
\end{align}
\end{proof}

\section{Proof of Theorem \ref{thm:joint}}\label{sec:proofoftheoremjoint}
\noindent
The (order-symmetric) joint eigenvalue density for the $N \times N$ GOE / GUE / GSE ensembles are given by the formula
\be
\rho(\lambda, G_\beta(N,q)) \ = \  C'_\beta(N,q) e^{-\beta q|\lambda|^2/4}|\Delta(\lambda)|^{\beta}
\ee
(see Theorem 3.1 of \cite{Rez}).  Here $\lambda = (\lambda_1,\ldots, \lambda_N)$ is the vector of eigenvalues, $C'_\beta(N,q)$ is a computable constant, and
\be
\Delta(\lambda) \ = \  \prod_{i<j} \left(\lambda_j - \lambda_i\right).
\ee
In other words, for a symmetric integrable function $f$ of $N$ real variables,
\be
\int_{G_\beta(N,q)} f(\lambda(A))d_{G_\beta(N,q)}A \ = \  \int_{\R^N} f(\lambda) \rho_N(\lambda, G_\beta(N,q)) d\lambda,
\ee
with $\rho(\lambda, G_\beta(N,q))$ as above. As by Lemma \ref{lem:map} the spherical ensemble measure is the pushforward of the Gaussian ensemble measure under the map $A \to \frac{\sqrt{r}}{|A|} A$, we have
\begin{align}
\int_{S_\beta(N,r)} f(\lambda(A))d_{S_\beta(N,r)}A
&\ = \  \int_{G_\beta(N,q)} f\left(\lambda\left(\frac{r}{|A|} A\right)\right)d_{G_\beta(N,q)}A \nn\\
&\ = \  \int_{G_\beta(N,q)} f\left( \frac{r}{|A|} \lambda\left(A\right)\right)d_{G_\beta(N,q)}A \nn\\
&\ = \  \int_{G_\beta(N,q)} f\left( \frac{r}{|\lambda(A)|} \lambda\left(A\right)\right)d_{G_\beta(N,q)}A \nn\\
&\ = \  \int_{\R^N} f\left(\frac{r}{|\lambda|} \lambda\right) \rho(\lambda,G_\beta(N,q)) d\lambda.
\end{align}
Changing to spherical coordinates, we obtain
\be\label{eq:sj}
\int_{S_\beta(N,r)} f(\lambda(A))d_{S_\beta(N,r)}A \ = \  \int_{rS^{N-1}} \int_0^\infty f\left(\theta\right) \rho\left(t\theta, G_\beta(N,q)\right) rt^{N-1}dtd\theta,
\ee
where $d\theta$ denotes the volume form of $rS^{N-1}$ (called $d\vec{\theta}$ earlier, in the $r=1$ case).
As $\Delta$ is a homogeneous polynomial in the eigenvalues of degree ${N \choose 2}$, we have $\rho(t\theta, G_\beta(N,q))
\ = \  C'_\beta(N,q) e^{-\beta q|t \theta|^2/4}|\Delta(t \theta)|^\beta
= C'_\beta(N,q) e^{-\beta qr^2t^2/4}t^{\beta {N \choose 2}}|\Delta(\theta)|^\beta$.
Substituting into \eqref{eq:sj},
\begin{align}
\int_{S_\beta(N,r)} f(\lambda(A))d_{S_\beta(N,r)}A
&\ = \  \int_{rS^{N-1}} \int_0^\infty f\left(\theta\right) \left(C'_\beta(N,q) e^{-\beta qr^2t^2/4}t^{\beta {N \choose 2}}|\Delta(\theta)|^\beta\right) rt^{N-1}dtd\theta\nn\\
&\ = \  C'_\beta(N,q)r \left(\int_0^\infty  t^{n-1}e^{-\beta qr^2t^2/4} dt\right)\left(\int_{rS^{N-1}} f\left(\theta\right) |\Delta(\theta)|^\beta d\theta\right),
\end{align}
where (as usual) $n = \beta {N \choose 2} + N$.
The substitution $u=\beta qr^2t^2/4$ immediately gives the value $\frac{2^{n-1}}{\left(\beta qr^2\right)^{n/2}} \Gamma\left(\frac{n}{2}\right)$ for the first integral.  Thus,
\be\label{eqn:nodelta}
\int_{S_\beta(N,r)} f(\lambda(A))d_{S_\beta(N,r)}A
\ = \  C''_\beta(N,r)\int_{rS^{N-1}} f\left(\theta\right)|\Delta(\theta)|^\beta d\theta,
\ee
where the normalizing constant
\be
C''_\beta(N,r) \ = \  C'_\beta(N,q)r \frac{2^{n-1}}{\left(\beta qr^2\right)^{n/2}} \Gamma\left(\frac{n}{2}\right)
\ee
is independent of $q$.  This equality \eqref{eqn:nodelta} is equivalent to Theorem \ref{thm:joint}: The Dirac delta in Theorem \ref{thm:joint} specifies that the mass is concentrated on $rS^{N-1}$. \hfill $\Box$

\section{Proof of Theorem \ref{thm:momentsandsemicircle}}\label{sec:proofofomentsandsemicircle}

\begin{proof}[Proof of Theorem \ref{thm:momentsandsemicircle}] We start by proving the moment expansion, equation \eqref{eq:momentsandsemicircletheoremmomentformula}.
Applied to the moments of the empirical spectral density, Lemma \ref{lem:map} implies
\begin{align}\label{eq:convert}
m_k(S_\beta(N,r)) &\ = \  \frac{1}{N} \int_{S_\beta(N,r)} \Tr(A^k)d_{S_\beta(N,r)}A\nn\\
                         &\ = \  \frac{1}{N} \int_{G_\beta(N,q)} \Tr\left(\left(\frac{r}{|A|}A\right)^k\right) d_{G_\beta(N,q)}A\nn\\
                         &\ = \  \frac{1}{N} C_{\beta}(N,q)r^k\int_{G_\beta(N,q)} \frac{\Tr\left(A^k\right)}{|A|^k} e^{-\beta q |A|^2/4}dA.
\end{align}
By Lemma \ref{lem:integration}, we may rewrite the above as
\be
m_k(S_\beta(N,r))
\ = \
\left(\frac{\beta q}{4}\right)^{\frac{k}{2}}\frac{\Gamma\left(\frac{n}{2}\right)}{\Gamma\left(\frac{k+n}{2}\right)}
\frac{1}{N} C_{\beta}(N,q)r^k\int_{G_\beta(N,q)} \Tr\left(A^k\right) e^{-\beta q |A|^2/4}dA,
\ee
where $n = \dim_\R G_\beta(N) = \beta {N \choose 2} + N$ (and we are using the fact that the norm of the Frobenius inner product is a real inner product). Thus,
\be
m_k(S_\beta(N,r))
\ = \
\left(\frac{\beta}{4} q r^2\right)^{\frac{k}2} \frac{\Gamma\left(\frac{n}{2}\right)}{\Gamma\left(\frac{k+n}{2}\right)}
m_k(G_\beta(N,q)),
\ee
as desired.

We now turn to the second part of the theorem, the convergence to the semi-circle. The eigenvalue density of $G_\beta(N,N)$ has mean zero and variance $1+\left(\frac{2}{\beta}-1\right)\frac{1}{N} = 1 + o(1)$.  The correct scaling for the spherical ensemble is to take $r = \sqrt{N}$, as this will lead to eigenvalues of size on the order of 1.  By \eqref{eq:momentsandsemicircletheoremmomentformula} we find
\begin{align}
m_k(S_\beta(N,\sqrt{N}))
&\ = \  \left(\frac{\beta}{4} N \sqrt{N}^2\right)^{\frac{k}2} \frac{\Gamma\left(\frac{n}{2}\right)}{\Gamma\left(\frac{k+n}{2}\right)} m_k(G_\beta(N,N))\nn\\
&\ = \  \left(\frac{\beta}{4} N^2\right)^{\frac{k}2} \frac{\Gamma\left(\frac{n}{2}\right)}{\Gamma\left(\frac{k+n}{2}\right)} m_k(G_\beta(N,N)),
\end{align}
where $n = \beta {N \choose 2} + N = \frac{N(\beta N+(2-\beta))}{2}$.  In particular, $S_\beta(N,\sqrt{N})$ has mean zero and variance
\begin{align}
m_2(S_\beta(N,\sqrt{N}))
&\ = \  \left(\frac{\beta}{4} N^2\right) \frac{\Gamma\left(\frac{n}{2}\right)}{\Gamma\left(1+\frac{n}{2}\right)} \left(1+\left(\frac{2}{\beta}-1\right)\frac{1}{N}\right)\nn\\
&\ = \  \left(\frac{\beta}{4} N^2\right) \left(\frac{1}{n/2}\right) \left(\frac{\beta N + (2-\beta)}{\beta N}\right)\nn\\
&\ = \  \left(\frac{\beta N^2}{4}\right) \left(\frac{4}{N(\beta N+(2-\beta))}\right) \left(\frac{\beta N + (2-\beta)}{\beta N}\right)\nn\\
&\ = \  1.
\end{align}
As $N \to \infty$, the asymptotic relation $\Gamma(z+c)/\Gamma(z) \sim z^c$ for $|z| \to \infty$ implies the asymptotic equality
\begin{align}
m_k(S_\beta(N,\sqrt{N}))
&\ \sim\  \left(\frac{\beta}{4} N^2\right)^{\frac{k}2} \left(\frac{n}{2}\right)^{-\frac{k}{2}} m_k(G_\beta(N,N))\nn\\
&\ \sim\  \left(\frac{\beta}{4} N^2\right)^{\frac{k}2} \left(\frac{\beta N^2}{4}\right)^{-\frac{k}{2}} m_k(G_\beta(N,N))\nn\\
&\ \sim\  m_k(G_\beta(N,N)).
\end{align}
In other words, if $C_\ell$ denotes the $\ell^{\rm th}$ Catalan number, then
\be
\twocase{\lim_{N \to \infty} m_k(S_\beta(N,\sqrt{N}))  \ = \  \lim_{N \to \infty} m_k(G_\beta(N,N)) \ = \  }{C_{k/2}}{if $k$ is even,}{0}{if $k$ is odd,} \ee
which by the Method of Moments (see \cite{Bi,Ta}) implies that, in the large $N$ limit, the empirical spectral distribution for a spherical ensemble converges in measure to the semicircle.  (See \cite{Rez}, Remark 2.2 for the fact that the limiting moments of Gaussian ensembles---and, more generally, Wigner ensembles---are Catalan numbers.)
\end{proof}

\section{Proof of Theorem \ref{thm:characteristicdensity}}\label{sec:proofoftheoremcharacteristicdensity}

\noindent
We first prove the expansion for the characteristic function in Theorem  \ref{thm:characteristicdensity}, and then we derive the claimed formula for the empirical spectral density.

The moments of the GUE are well-known.  The moment $m_2(S_\beta(N,r))$ vanishes if $k$ is odd.  For $k=2\ell$ we have the following expression from Harer and Zagier \cite{HZ}:
\be\label{eq:expansionGtwoNq}
m_{2\ell}(G_2(N,q)) \ = \  \frac{1}{N}q^{-\ell}(2\ell-1)!!c(\ell,N),
\ee
where $c(\ell, N)$ are power series coefficients for a particular rational function,
\be\label{eq:rat}
\sum_{\ell=0}^\infty c(\ell,N)z^\ell \ = \  \frac{1}{2z}\left(\left(\frac{1+z}{1-z}\right)^N-1\right),
\ee
and $(2\ell-1)!! = (2\ell-1)(2\ell-3)\cdots 3 \cdot 1$ is the $2\ell^{\rm th}$ moment of the standard normal. Substituting \eqref{eq:expansionGtwoNq} into \eqref{eq:momentsandsemicircletheoremmomentformula} yields
\begin{align}
m_{2\ell}(S_2(N,r)) &\ = \  \frac{1}{N}\left(\frac{r^2}{2}\right)^\ell \frac{\Gamma\left(\frac{n}{2}\right)}{\Gamma\left(\frac{n}{2}+\ell\right)} (2\ell-1)!! c(\ell,N),
\end{align}
and we may therefore write the characteristic function as
\begin{align}
\phi(t,S_2(N,r)) &\ = \  \sum_{\ell = 0}^\infty \left(\frac{1}{N}\left(\frac{r^2}{2}\right)^\ell \frac{\Gamma\left(\frac{n}{2}\right)}{\Gamma\left(\frac{n}{2}+\ell\right)} (2\ell-1)!! c(\ell,N)\right) \frac{(it)^{2\ell}}{(2\ell)!} \nn\\
&\ = \  \frac{1}{N} \sum_{\ell = 0}^\infty \frac{\Gamma\left(\frac{n}{2}\right)}{\Gamma\left(\frac{n}{2}+\ell\right) \ell!} c(\ell,N) \left(-\frac{1}{4} (rt)^2\right)^\ell.
\end{align}
For convenience, fix $N$ and $r$, and set
\be
\psi_N(u) \ := \  \frac{1}{N} \sum_{\ell = 0}^\infty \frac{\Gamma\left(\frac{n}{2}\right)}{\Gamma\left(\frac{n}{2}+\ell\right) \ell!} c(\ell,N) u^\ell,
\ee
so that $\psi_N\left(-\frac{1}{4} (rt)^2\right)=\phi(t,S_2(N,r))$.  The coefficients of this power series are a product of more familiar power series coefficients:
\be
\sum_{\ell=0}^\infty \frac{\Gamma\left(\frac{n}{2}\right)}{\Gamma\left(\frac{n}{2}+\ell\right) \ell!} z^\ell \ = \  \hf\left(\frac{n}{2},z\right),
\ee
where $\hf$ is a confluent hypergeometric limit function, and by definition
\be
\sum_{\ell=0}^\infty c(\ell,N)z^\ell \ = \  B(z)
\ee
where $B(z)$ is the rational function on the right hand side of \eqref{eq:rat}. From contour integration, we find
\be\label{eq:contourintegralforpsiNu}
\psi_N(u) \ = \  \frac{1}{2\pi i N}\oint_{|z|=2} \hf\left(\frac{n}{2},uz\right)B(z^{-1})\frac{dz}{z}.
\ee
We evaluate the contour integral by expanding both functions at $z=1$, which is the location of the pole of $B(z^{-1})$.  The expansion of $B$ is easy:
\begin{align}
B(z^{-1})\frac{dz}{z} \ = \  \foh\left(\left(1+\frac{2}{z-1}\right)^N-1\right) dz\ = \  \foh\sum_{\ell=1}^N {N \choose \ell} \left(\frac{2}{z-1}\right)^\ell dz.
\end{align}
The expansion of $\hf$ is more involved but elementary:
\begin{align}
_0F_1\left(\frac{n}{2}, uz\right) &\ = \  \sum_{\ell = 0}^\infty \frac{\Gamma\left(\frac{n}{2} \right)}{\Gamma\left(\frac{n}{2} + \ell\right)\ell!} (uz)^\ell \nn\\
&\ = \  \sum_{\ell=0}^\infty \frac{\Gamma\left(\frac{n}{2}\right)}{\Gamma\left(\frac{n}{2} + \ell\right)\ell!} u^\ell \left((z-1) + 1\right)^\ell \nn\\
&\ = \  \sum_{\ell = 0}^\infty \frac{\Gamma\left(\frac{n}{2}\right)}{\Gamma\left(\frac{n}{2} + \ell\right)\ell!} u^\ell \sum_{j=0}^\infty {\ell \choose j} (z-1)^j \nn\\
&\ = \  \sum_{j=0}^\infty  \frac{\Gamma\left(\frac{n}{2}\right)}{j!} \left( \sum_{\ell = j}^\infty \frac{1}{\Gamma\left(\frac{n}{2} + \ell\right)  (\ell - j)!} u^\ell \right) (z-1)^j \nn\\
&\ = \  \sum_{j=0}^\infty  \frac{\Gamma\left(\frac{n}{2}\right)}{\Gamma\left(\frac{n}{2} + j\right) j!} \left( \sum_{\ell' = 0}^\infty \frac{\Gamma\left(\frac{n}{2} + j\right)}{\Gamma\left(\frac{n}{2} + j+\ell'\right)  \ell'!} u^{\ell'+j} \right) (z-1)^j \nn\\
&\ = \  \sum_{j=0}^\infty  \frac{\Gamma\left(\frac{n}{2}\right)}{\Gamma\left(\frac{n}{2} + j\right) j!}  {}_0F_1\left(\frac{n}{2}+j,u\right) u^j (z-1)^j.
\end{align}
By the residue theorem, the only surviving terms after the integration are those with $j-\ell=-1$.  Precisely, the integral in \eqref{eq:contourintegralforpsiNu} evaluates to
\begin{align}
\psi_N(u) &\ = \  \frac{1}{2N} \sum_{j=0}^{N-1} {N \choose j+1} 2^{j+1} \frac{\Gamma\left(\frac{n}{2}\right)}{\Gamma\left(\frac{n}{2} + j\right) j!}  \hf\left(\frac{n}{2}+j,u\right) u^j\nn\\
&\ = \  \frac{1}{N} \sum_{j=0}^{N-1} {N \choose j+1} \frac{\Gamma\left(\frac{n}{2}\right)}{\Gamma\left(\frac{n}{2} + j\right) j!}  \hf\left(\frac{n}{2}+j,u\right) (2u)^j,
\end{align} and thus
\begin{align}
\phi(t,S_2(N,r)) &\ = \  \psi_N\left(-\fof (rt)^2\right)\nn\\
                        &\ = \  \frac{1}{N} \sum_{j=0}^{N-1}  {N \choose j+1} \frac{\Gamma\left(\frac{n}{2}\right)}{\Gamma\left(\frac{n}{2} + j\right) j!}  \hf\left(\frac{n}{2}+j,\fof (rt)^2\right) \left(-\foh (rt)^2\right)^j.
\end{align}
The expression in terms of Bessel functions in Theorem \ref{thm:characteristicdensity} follows immediately, using the formula (from \cite{AS}, p. 362)
\be
J_\alpha(t) \ = \  \frac{(t/2)^\alpha}{\Gamma(\alpha+1)}\hf\left(\alpha+1,-\fof t^2\right).
\ee
This completes the proof of the first part of Theorem \ref{thm:characteristicdensity}.

To compute the density function, we continue to use confluent hypergeometric limit functions because they have simple Fourier transforms.
\begin{lem}\label{lem:hfFourier}
The Fourier transform of $\hf\left(\alpha,-\fof t^2\right)$ is given by
\be
\frac{1}{2\pi}\int_{-\infty}^{\infty} e^{-i tx}\hf\left(\alpha,-\fof t^2\right)\,dt \ = \  \foh {\alpha -1 \choose \foh}\left(1-x^2\right)^{\alpha-\frac32},
\ee where the binomial coefficient is extended so that ${w \choose z} := \frac{\Gamma(w+1)}{\Gamma(z+1)\Gamma(w-z+1)}$, a meromorphic function of two complex variables $w$ and $z$.
\end{lem}
\begin{proof}
Fix $\alpha$, and let $w(t) = \hf\left(\alpha,-\fof t^2\right)$.  From the power series expansion of $w$, it's easy to check that $w$ satisfies the differential equation
\begin{equation}
tw''(t)+(2\alpha - 1)w'(t)+tw(t)=0.
\end{equation}
Taking the Fourier transform of both sides, and using the rules
\begin{align}
\fF\{tg(t)\}&=i\hat{g}'(x),\\
\fF\{g'(t)\}&=(ix)\hat{g}(x),
\end{align}
we obtain the following differential equation for $\hat{w}$:
\begin{equation}
(1-x^2)\hat{w}'(x)+(2\alpha-3)x\hat{w}(x)=0.
\end{equation}
In other words, $\frac{d}{dx}\left(\log\hat{w}(x)\right) = (\alpha-\frac32)\frac{-2x}{1-x^2}$, so $\hat{w}(x)=C(1-x^2)^{\alpha-\frac32}$ for some constant $C$.

The formula
\begin{equation}
C = \frac{1}{2\pi}\int_{-\infty}^{\infty} \hf\left(\alpha,-\fof t^2\right)\,dt = \foh {\alpha -1 \choose \foh}
\end{equation}
follows from the formulas for the integrals of Bessel functions in Chapter 11 of \cite{AS}.
\end{proof}

It follows from Lemma \ref{lem:hfFourier} that
\begin{align}
\frac{1}{2\pi}\int_{-\infty}^{\infty} e^{-i tx}\hf\left(\frac{n}{2}+j,-\fof t^2\right)\left(-\foh t^2\right)^j\,dt
&\ = \  \frac{1}{2^{j+1}} {\frac{n}{2}+j-1 \choose \foh}\left(\frac{d}{dx}\right)^{2j}\left(1-x^2\right)^{\frac{n-3}{2}+j} \end{align}
and
\begin{eqnarray} & & \frac{1}{2\pi}\int_{-\infty}^{\infty} e^{-i tx}\hf\left(\frac{n}{2}+j,-\fof (rt)^2\right)\left(-\foh (rt)^2\right)^j\,dt
\nonumber\\ & &\ \ \ \ \ \ \ = \  \frac{1}{2^{j+1}r} {\frac{n}{2}+j-1 \choose \foh}\left(\frac{d}{d(x/r)}\right)^{2j}\left(1-(x/r)^2\right)^{\frac{n-3}{2}+j} \nonumber\\ & &\ \ \ \ \ \ \ = \
\frac{1}{2^{j+1} r^{n-2}}{\frac{n}{2}+j-1 \choose \foh}\left(\frac{d}{dx}\right)^{2j}\left(r^2-x^2\right)^{\frac{n-3}{2}+j}.
\end{eqnarray}

From Fourier inversion we finally obtain
\begin{align}
f\left(x,S_2(N,r)\right)
&\ = \  \frac{1}{2\pi}\int_{-\infty}^{\infty} e^{-i tx}\phi(t,S_2(N,r))\,dt \nn\\
&\ = \  \frac{1}{N}\sum_{j=0}^{N-1} {N \choose j+1} \frac{\Gamma\left(\frac{n}{2}\right)}{\Gamma\left(\frac{n}{2} + j\right) j!}\frac{1}{2^{j+1} r^{n-2}}{\frac{n}{2}+j-1 \choose \foh}\left(\frac{d}{dx}\right)^{2j}\left(r^2-x^2\right)^{\frac{n-3}{2}+j} \nn\\
&\ = \  \frac{1}{r^{n-2}N}\sum_{j=0}^{N-1} {N \choose j+1} \frac{\Gamma\left(\frac{n}{2}\right)}{2^{j}\sqrt{\pi}\Gamma\left(\frac{n-1}{2}+j\right) j!}\left(\frac{d}{dx}\right)^{2j}\left(r^2-x^2\right)^{\frac{n-3}{2}+j},
\end{align}
and substituting $n=N^2$ completes the proof of Theorem \ref{thm:characteristicdensity}. \hfill $\Box$

In Figures \ref{fig:DensityPlots468} and \ref{fig:DensityPlots91625}, we plot the density $f(x,S_2(N,\sqrt{N}))$ for various values of $N$. (Note the polynomial expansions are valid for all $N$, though if $N$ is a perfect square, then $r = \sqrt{N}$ is an integer.) The convergence to the semicircle is quite rapid, and is similar to the convergence to the semicircle found in \cite{KKMSX} for block circulant ensembles.

\begin{figure}
\begin{center}
\scalebox{.58}{\includegraphics{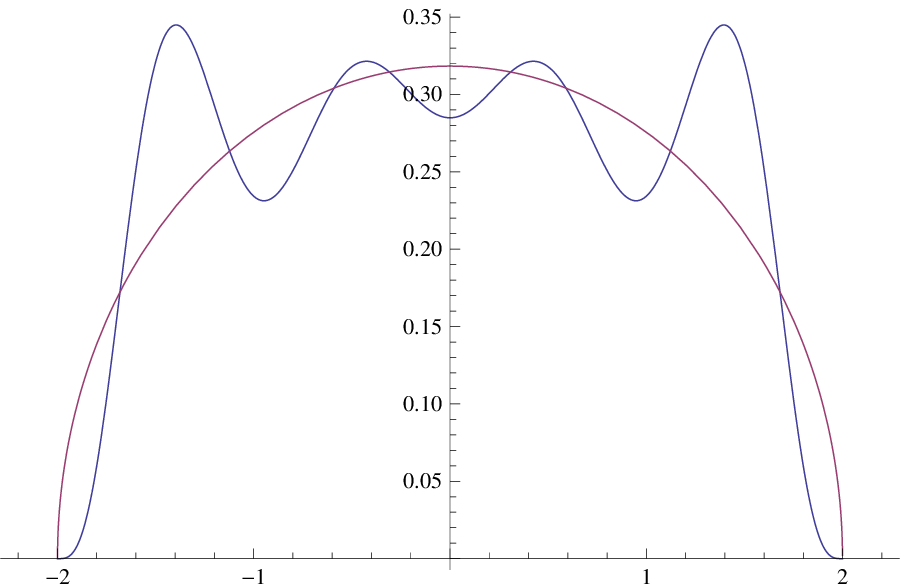}}\ \scalebox{.58}{\includegraphics{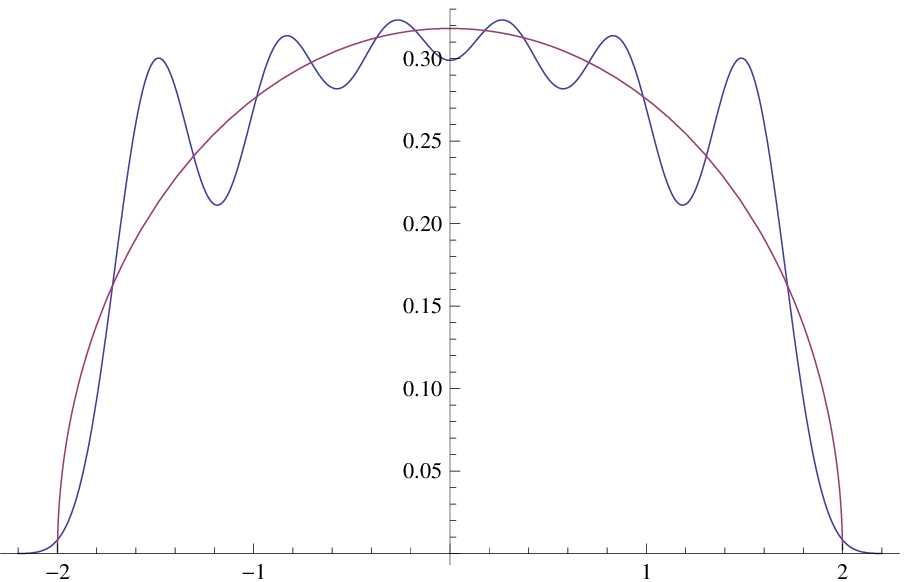}} \ \scalebox{.58}{\includegraphics{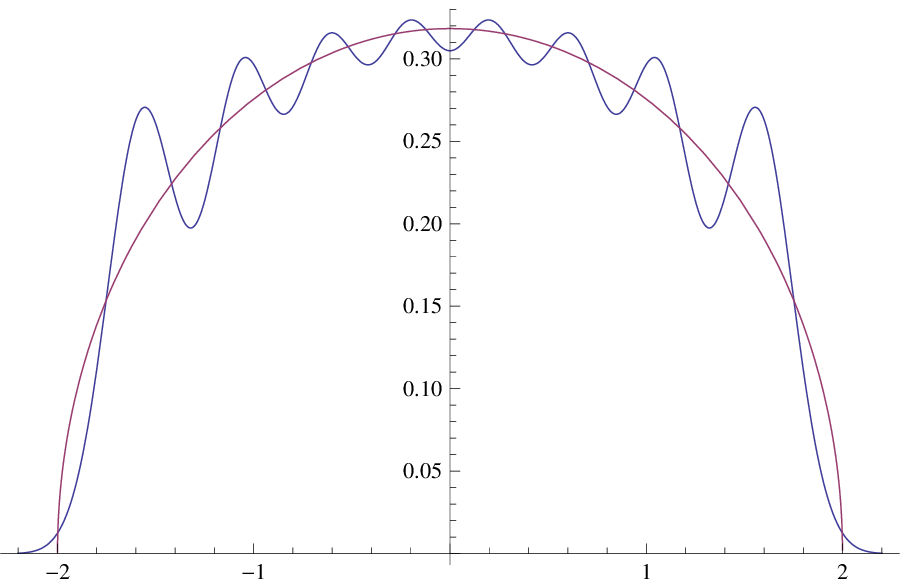}}
\caption{\label{fig:DensityPlots468} Plots of the density $f(x,S_2(N,\sqrt{N}))$ for $N = 4, 6$ and $8$.}
\end{center}\end{figure}

\begin{figure}
\begin{center}
\scalebox{.58}{\includegraphics{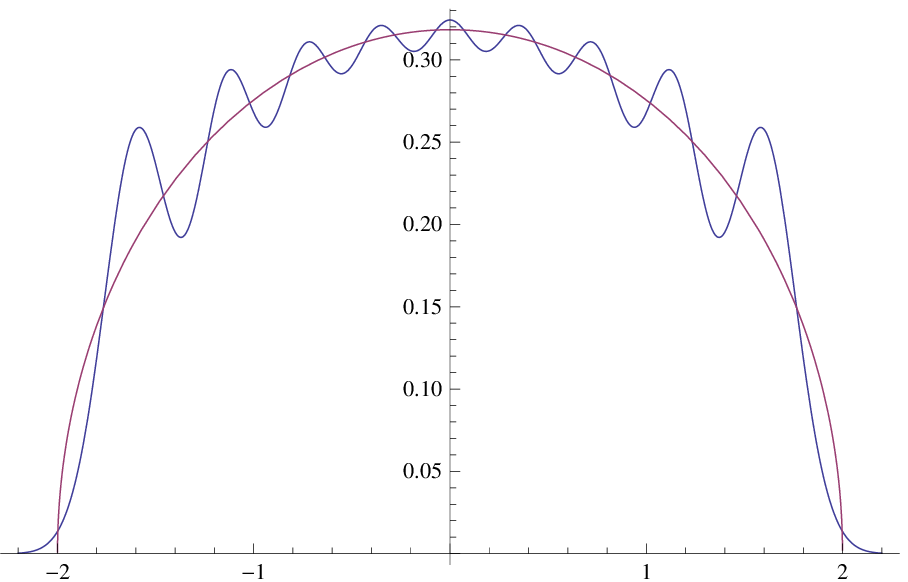}}\ \scalebox{.58}{\includegraphics{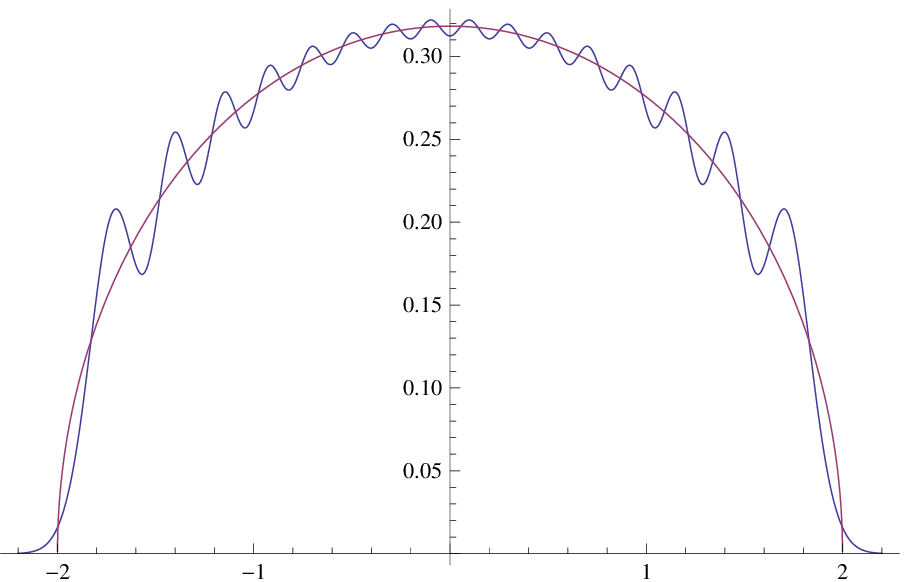}} \ \scalebox{.58}{\includegraphics{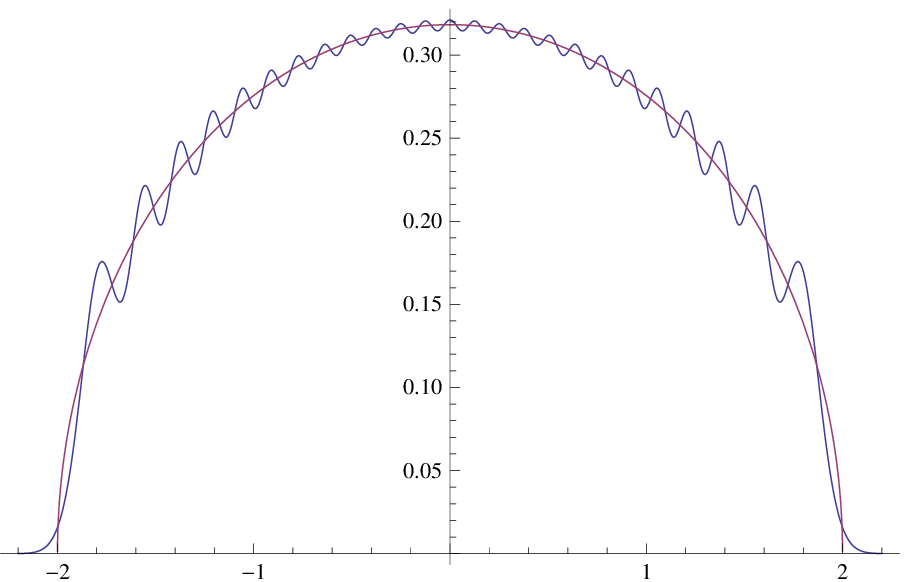}}
\caption{\label{fig:DensityPlots91625} Plots of the density $f(x,S_2(N,\sqrt{N}))$ for $N = 9, 16$ and $25$.}
\end{center}\end{figure}





\section{Spacing Statistics \label{sec:spacings}}

%

\noindent
We conclude with a statistical observation. The spacing density of the SUE appears to approach the same large $N$ limit as that of the GUE.  This is not surprising, because spacing densities are considered to be very robust.  For example, the thin ensemble of real symmetric Toeplitz matrices appears to have GOE spacings, even though its limiting empirical density is far from semicircular \cite{HM}.  Spacing statistics for the SUE and GUE are compared in Figure \ref{fig:SpacingPlots}.

\begin{figure}
\begin{center}
\scalebox{.58}{\includegraphics{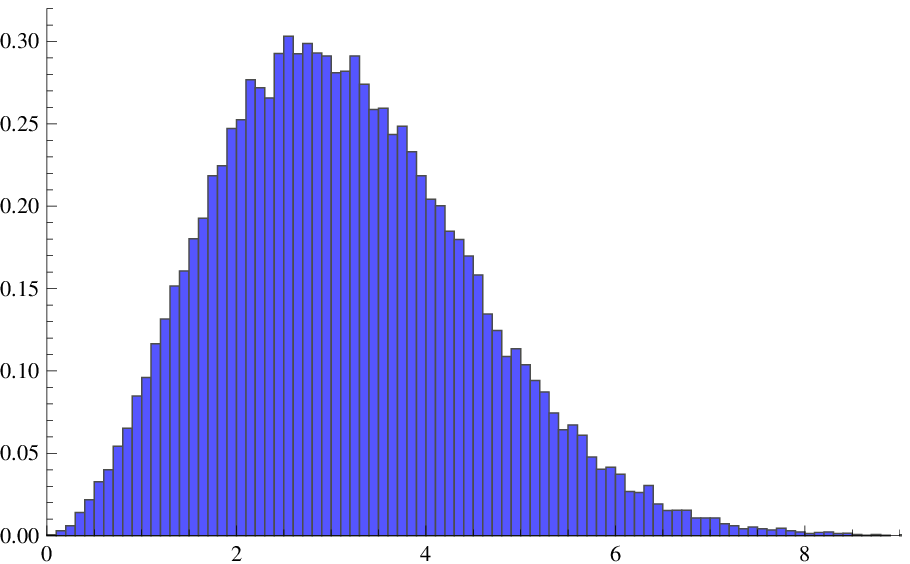}}\ \scalebox{.58}{\includegraphics{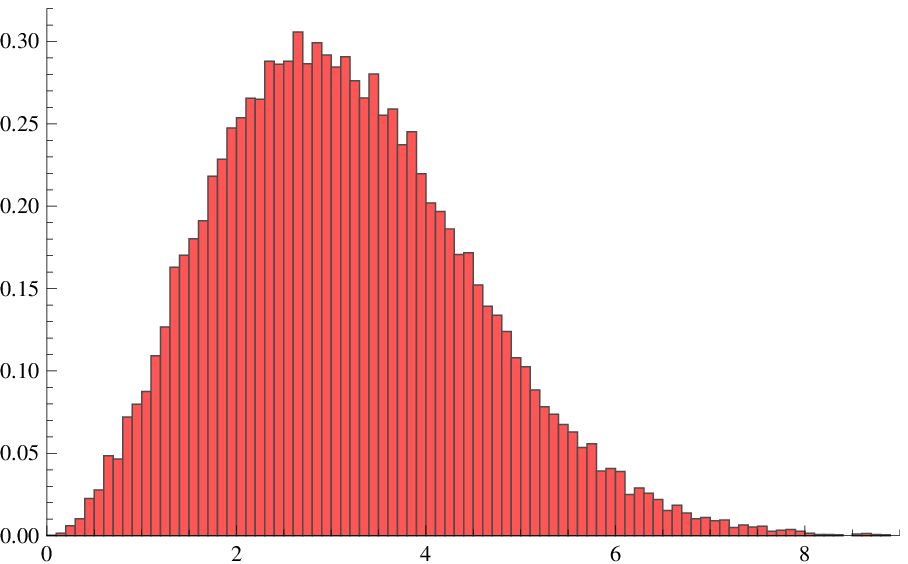}} \ \scalebox{.58}{\includegraphics{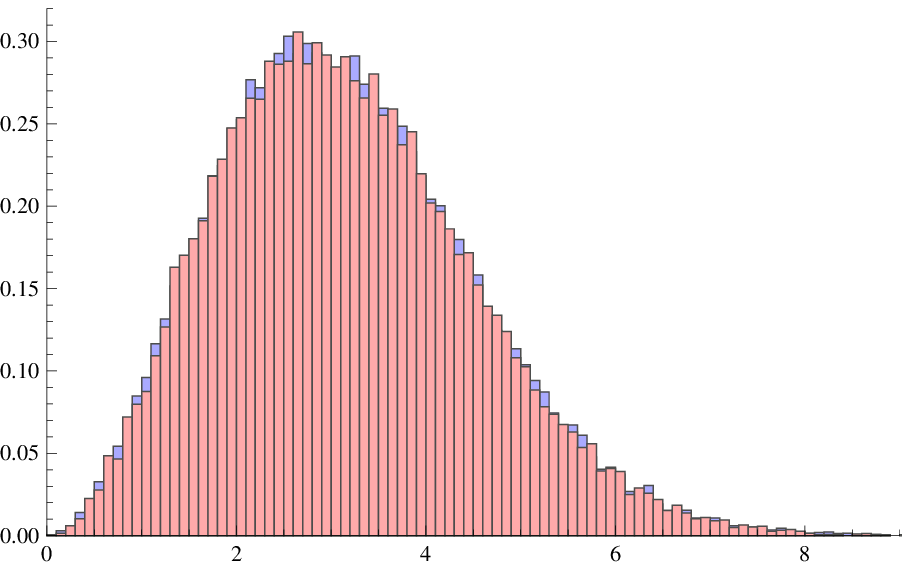}}
\caption{\label{fig:SpacingPlots} Histograms showing spacing statistics for the SUE, GUE, and the SUE and GUE together, respectively.  For both ensembles, $2000$ matrices of size $100 \times 100$ were sampled, and the $20$ consecutive spacings of the middle $21$ eigenvalues were computed.}
\end{center}\end{figure}



\ \\

\end{document}